\documentclass[letterpaper, 10 pt, conference]{IEEEconf}  

\IEEEoverridecommandlockouts                              
\usepackage{amsmath} 
\usepackage{amssymb}  
\usepackage{amsfonts}

\usepackage{graphicx} 

\usepackage{ifthen}
\DeclareMathAlphabet{\mbb}{U}{bbold}{m}{n} 
\newcommand \mb[1] {\ifthenelse{\equal{#1}{0}}{\mbb{0}}{\ifthenelse{\equal{#1}{1}}{\mbb{1}}{\mathbb{#1}}}} 

\newcommand \mc {\mathcal}

\newcommand \RR {\mb R}

\newcommand \inv {^{-1}}
\newcommand \T {^\top}
\newcommand \invT {^{-1\top}}

\newcommand \sdots {\text{\scriptsize{\dots}}}
\newcommand \diag {\mathrm{diag}}

\newcommand \ie {\textit{i.e. }}

\newcommand{\on}[1]{\operatorname{#1}}

\newcommand{\set}[2]{ \left\{ #1 \left|~ \vphantom{#1} #2 \right. \right\} }
\renewcommand \subset {\subseteq}

\renewcommand{\l}{\left}
\renewcommand{\r}{\right}

\usepackage{amsthm} 
\newtheorem{theorem}{Theorem}

\newtheorem{lemma}[theorem]{Lemma}

\newtheorem{assumption}[theorem]{Assumption}
\theoremstyle{definition}

\newtheorem{defn}[theorem]{Definition}

\newtheorem{example}[theorem]{Example}

\title{\LARGE \bf
Plug-and-play Solvability of the Power Flow Equations for Interconnected DC Microgrids with Constant Power Loads
}

\author{Mark Jeeninga, Claudio De Persis and Arjan J. van der Schaft
\thanks{*This work is supported by the NWO (Netherlands Organisation for Scientific Research) project 'Energy management strategies for interconnected smart microgrids' within the DST-NWO Joint Research Program on Smart Grids.}
\thanks{$^{1}$Mark Jeeninga is with the Engineering and Technology Institute Groningen, and the Bernoulli Institute for Mathematics, Computer Science and Artificial Intelligence, University of Groningen, 9747AG Groningen, The Netherlands
{\tt\small m.jeeninga@rug.nl}}%
\thanks{$^{2}$Claudio De Persis is with Engineering and Technology Institute Groningen, University of Groningen, 9747AG Groningen, The Netherlands
{\tt\small c.de.persis@rug.nl}}%
\thanks{$^{3}$Arjan J. van der Schaft is with the Bernoulli Institute for Mathematics, Computer Science and Artificial Intelligence, University of Groningen, 9747AG Groningen, The Netherlands
{\tt\small a.j.van.der.schaft@rug.nl}}%
}

\begin{document}

\maketitle
\thispagestyle{empty}
\pagestyle{empty}

\begin{abstract}
In this paper we study the DC power flow equations of a purely resistive DC power grid which consists of interconnected DC microgrids with constant-power loads.
We present a condition on the power grid which guarantees the existence of a solution to the power flow equations. 
In addition, we present a condition for any microgrid in island mode which guarantees that the power grid remains feasible upon interconnection.
These conditions provide a method to determine if a power grid remains feasible after the interconnection with a specific microgrid with constant-power loads.

Although the presented condition are more conservative than existing conditions in the literature, its novelty lies in its plug-and-play property. 
That is, the condition gives a restriction on the to-be-connected microgrid, but does not impose more restrictions on the rest of the power grid.
\end{abstract}

\section{Introduction}
The increasing use and demand of electricity is pushing our power grids to their limits. 
In addition, the incorporation of renewable energy sources is steadily introducing more uncertainty into our energy sources.
It has therefore become a necessity to better understand the fundamental limits of our power grid, as well as to come up with smart ways to generate and distribute power.

To address these issues, there has been an increasing interest in the study of microgrids. A microgrid is a small power grid within a greater power grid.
One of their key features is their plug-and-play capability: a microgrid can disconnect from the main grid to operate in island mode, and reconnect whenever necessary. 
This allows for leveling out the fluctuations of power generation and consumption, as well as for a microgrid to disconnect when the main grid suffers from a power outage. 

There has been an increasing interest in DC (direct current) microgrids. A large portion of the day-to-day energy consumption is converted to DC. Together with the rise of photo-voltaic cells and advances in battery storage, it makes more and more sense to implement DC power grids on a larger scale.

The matching of supply and demand of power in a power grid is governed by its power flow equations.
If these equations have no solution, long-term voltage stability is lost and phenomena such as blackouts and voltage collapse occur \cite{paper13}.
The introduction of constant-power loads leads to non-linearities for which the power flow equations may not be solvable.
Conditions which guarantee their solvability are known, but are considered conservative \cite{paper4,paper12,paper8}.

There has been some research on the control of DC microgrids.
A distributed control scheme was proposed in \cite{paper15}, and a plug-and-play control scheme was proposed in \cite{paper5}, although neither paper consider constant-power loads.
To the best of our knowledge, no theoretical advances have been made towards power flow solvability where new sources and/or constant-power loads are introduced, and the best available approach is to recalculate the known conditions for the altered power grid.

To this end, the goal of this paper is to give conditions on the power grid and a to-be-attached microgrid, not assuming any control scheme, which guarantees that (i) the power flow equations of the power grid are solvable, (ii) the condition for (i) also holds for the interconnection of the power grid with the microgrid which makes the power flow equations of the interconnection solvable, and (iii) the conditions for (i) still hold for the original power grid after the microgrid is connected.
The main advantage of this approach is that, as new microgrids are attached to the power grid, reverification of condition (i) is not necessary. We refer to this as the \emph{plug-and-play property}.

The novelty of the presented approach is that it uses the (block) Cholesky decomposition as a theoretical tool to analyze the power flow equations. Commonly, whenever a new load is introduced to a power grid, conditions for the feasibility of the power grid should be re-evaluated. To suppress this effect we restrict to a ``directional'' power flow condition, which is based on the block Cholesky decomposition, where each block represents a microgrid.

%
%



\subsection{Organization of the Paper}
This paper is organized as follows. Section \ref{prelim section} establishes notation, lists a number of graph theoretic notions and reviews M-matrices. 
Section \ref{power flow section} deals with the feasibility of a DC power grid. In Section \ref{interconnection section} we analyze the interconnection of multiple microgrids and review the block Cholesky decomposition. From this we derive a sufficient condition for solvability of the power flow equations. In Section \ref{dynamic section} we consider the process of interconnecting a microgrid to a power grid, and state our main result. Section \ref{conclusion section} concludes the paper.

\section{Preliminaries}\label{prelim section}
\subsection{Notation} All vectors and matrices in this paper are real-valued. We denote a diagonal matrix by $[d]:=\diag(d_1,\sdots,d_n)$ for a vector $d$. Note that $[v]w = [w]v$ for any pair of vectors $v,w$. Inequalities between matrices such as $A>0$ and $A\ge 0$ are meant element-wise, and the same holds for vectors. The notation $\mb 0$ and $\mb 1$ is used for all-zeros and all-ones vectors, respectively.
We let $I_n$ denote the $n\times n$ identity matrix. 
\begin{defn}
We say $A$ is \emph{order-preserving} if $x<y$ implies $Ax<Ay$, for all vectors $x,y$. 
\end{defn}
It can be shown that $A$ is order-preserving if and only if $A\ge 0$ and $A\mb 1 > \mb 0$.
Any nonnegative invertible matrix is order-preserving, and its inverse is known in the literature as a \emph{monotone matrix} \cite{book3}.
The sum and product of two order-preserving matrices is again order-preserving.

\subsection{Graph Theory}
We let it be understood that by a graph we mean an undirected weighted graph with no self-loops. That is, a graph $\Gamma$ is the tuple $(\mc V,\mc E,\on w)$ with node set $\mc V = \{1,\sdots,|\mc V|\}$, edge set $\mc E\subset \set{\{i,j\}}{i,j\in \mc V, i\neq j}$, and a map $\on w: \mc E\to \RR_{>0}$ of edge weights.

For any $\tilde {\mc V}\subset \mc V$, the subgraph induced by $\tilde {\mc V}$ is the tuple $(\tilde {\mc V},\tilde {\mc E},\on w)$ where $\tilde {\mc E} = \set{\{i,j\}\in \mc E}{i,j\in \tilde {\mc V}}$.

We say that two nodes $i,j\in \mc V$ are path-connected with respect to the graph $\Gamma=(\mc V,\mc E,\on w)$ if there exists a path $(v_1,v_2,\sdots,v_l)$ such that $\{v_r,v_{r+1}\}\in \mc E$ for $r=1,\sdots,l-1$, $i=v_1$ and $j=v_l$.

The node set $\mc V$ of a graph can be partitioned into connected components such that two nodes are in the same connected component if and only if they are path-connected with respect to that graph.

We let the Laplacian matrix $\on L(\Gamma)$ of the graph $\Gamma=(\mc V,\mc E,\on w)$ be defined by $\on L(\Gamma)_{ij} = -\on w(\{i,j\})$ if $\{i,j\}\in \mc E$, $\on L(\Gamma)_{ij} =0$ if $i\neq j$ and $\{i,j\}\not\in \mc E$, and $\on L(\Gamma)_{ii} = -\sum_{j\neq i}\on L(\Gamma)_{ij}$.
Each graph is fully determined by its Laplacian matrix.

\subsection{M-matrices}
We define an M-matrix as follows.
\begin{defn}
A matrix $A$ is an M-matrix if there exists a matrix $B\ge 0$ such that $A = sI-B$ where $s\ge\rho(B)$, the spectral norm of $B$. It is invertible if $s>\rho(B)$.
\end{defn}
The diagonal elements of an M-matrix are nonnegative, and its off-diagonal elements are nonpositive.

If an M-matrix is invertible, its inverse is nonnegative. Any Schur complement of an invertible M-matrix is again an invertible M-matrix.\footnote{See \cite{book3}, Theorem 6.2.3 and Exercise 10.6.1.} 

\section{Feasibility of DC Power Grids}\label{power flow section}
We consider a purely resistive DC grid where $\begin{pmatrix}
V_L\\V_S
\end{pmatrix}>\mb 0$ and $\begin{pmatrix}
I_L\\I_S
\end{pmatrix}$ denote the voltages potentials and currents at the nodes, respectively, which are real-valued and partitioned according to nodes being loads ($ L $) or sources ($ S $).
The power at each node is given by 
\begin{align*}
\begin{pmatrix}
- P_L\\P_S
\end{pmatrix} = \begin{bmatrix}
\begin{pmatrix}
V_L\\V_S
\end{pmatrix}
\end{bmatrix}
\begin{pmatrix}
I_L\\I_S
\end{pmatrix},
\end{align*}
where $P_L,P_S>\mb 0$, indicating that load nodes consume power, while source nodes supply power.
We consider the setting where each load is a constant-power load, and therefore take $P_L$ constant.
The voltage and current in a resistive circuit are related via the weighted Laplacian (Kirchhoff) matrix $Y=\begin{pmatrix}
Y_{LL} & Y_{LS}\\
Y_{SL} & Y_{SS}
\end{pmatrix}$, with the weights corresponding to the admittance of the lines between the nodes in the power grid.\footnote{While we do not consider shunts in our power grids, all results presented in this paper hold also for loads with shunts, for which one would use so-called grounded Laplacian matrices.}

The current at the load nodes flowing into the power grid is given by $I_L = Y_{LL}V_L + Y_{LS}V_S < \mb 0$.
To satisfy the power demand of the loads, the following equation, known as the (purely resistive) DC power flow equation, should be satisfied for $V_L>\mb 0$, where $P_L$ and $V_S$ are given:
\begin{align}\label{main problem}
[V_L]Y_{LL}V_L + [V_L] Y_{LS}V_S + P_L = \mb 0.
\end{align}
To simplify notation, we introduce the following definition.
\begin{defn}
Let $A$ be a square matrix and vectors $b\ge \mb 0,c>\mb 0$.
We say $(A,b,c)$ is \emph{feasible} if there exists a vector $x>\mb 0$ such that $[x]Ax - [b]x + c = \mb 0$.
\end{defn}

In addition, we say the power grid represented by $Y$ is \emph{feasible} if $(Y_{LL},-Y_{LS}V_S,P_L)$ is feasible, where $V_S$ and $P_L$ are given. Note that feasibility does not depend on $Y_{SS}$.

Being a quadratic equation in the entries of $V_L$, (\ref{main problem}) is not always solvable. A necessary and sufficient condition for solvability is not known when multiple interconnected load nodes are considered.\footnote{If there are no connections among load nodes, $Y_{LL}$ is diagonal and each row of (\ref{main problem}) is a quadratic equation in $V_{L,i}$. The discriminant associated to each row is nonnegative if and only if there exists a positive solution to \eqref{main problem}.} 

%
The paper \cite{paper1} gives the following sufficient condition for the existence of a solution to (\ref{main problem}), albeit in a different form and context.

\begin{theorem}[Simpson-Porco \textit{et al.}]\label{SP result}
If $Y_{LL}\inv[V_L^*]\inv P_L< \frac 1 4 V_L^* $, then $(Y_{LL},-Y_{LS}V_S,P_L)$ is feasible, where $V_L^* := -Y_{LL}\inv Y_{LS} V_S$ are the open-circuit voltages. 
\end{theorem}
The open-circuit voltages correspond to the voltage potentials in the situation where there is no power demand and therefore $I_L=\mb 0$.



\section{Microgrids with a fixed\break interconnection topology}\label{interconnection section}

Throughout the rest of this paper we consider a DC power grid consisting of multiple interconnected DC microgrids.
Our aim is to study the related power flow problem, while respecting the topological structure of microgrids. In particular, we phrase a sufficient condition for (\ref{main problem}) for such a power grid in terms of conditions for individual microgrids and
according to a hierarchical structure. This hierarchical structure represents the order in which the microgrids are connected to the power grid. To clarify, microgrids which are lower in the hierarchy were added before the ones that are higher in the hierarchy.

 
The presented approach assumes that sources do not supply power to load nodes which are lower in the hierarchy, and therefore leads to more conservative conditions.
On the other hand, the main feature of the approach is that only considering power flow between microgrids in a ``specified direction" gives conditions for which the introduction of a new microgrid does not alter the conditions on the original power grid.

The main result of this paper is derived from Theorem \ref{SP result} and relies on the block Cholesky decomposition (BCD) to describe the hierarchical structure between the microgrids. The BCD is reviewed in Section \ref{bcd section}.

Throughout the rest of this section we focus on a fixed topology of microgrids and study their feasibility. In Section \ref{dynamic section} we consider a dynamic topology of microgrids, where the introduction and interconnection of a new microgrid to a power grid is studied.

\subsection{System description}\label{system description}
We consider a DC power grid as in Section \ref{power flow section}, where again $Y=\begin{pmatrix}
Y_{LL} & Y_{LS}\\
Y_{SL} & Y_{SS}
\end{pmatrix}$. Our power grid is subdivided into $k$ microgrids. We number our microgrids and denote the $i$-th microgrid by $\mc M_i$. The numbering of the microgrids represents the order in which the microgrids were attached to the power grid, which induces a hierarchical structure on the microgrids.




We let $\Gamma = (\mc V,\mc E,\on w)$ to be the graph represented by $Y$. The nodes of a microgrid $\mc M_i$ induce a subgraph of $\Gamma$. Such a subgraph is denoted by $\Gamma_{\mc M_i}$.

In order to state the main result of this section, we require the following assumption on the microgrids.
\begin{assumption}\label{as}
For each microgrid $\mc M_i$, each load node in $\mc M_i$ satisfies at least one the following conditions:
\begin{enumerate}
\item[i)] The node is path-connected with respect to the graph $\Gamma_{\mc M_i}$ to a source node in $\mc M_i$;
\item[ii)] The node is path-connected with respect to the graph $\Gamma$ to a node in $\mc M_j$ with $j<i$.
\end{enumerate}
\end{assumption}
These conditions can be interpreted as follows: Condition \ref{as}.i is equivalent to the condition that the node is path-connected to a source node in $\mc M_i$ when the microgrid operates in island mode. Condition \ref{as}.ii is equivalent to the condition that every load node is path-connected to a node lower in the hierarchy.
Note that Assumption \ref{as} only requires that the first microgrid is able to operate in island mode. In addition, note that we do not assume that the power grid is connected.

Assumption \ref{as} prevents the case where a microgrid is connected to a power grid while some of its load nodes are not path-connected to a source node; such a case can never be feasible. More precisely, it states that for each load node there exists a path $(v_1,\sdots,v_l)$ to a source node such that it descends hierarchy of microgrids; that is, if $v_i\in \mc M_s$ and $v_j\in \mc M_t$, then $i<j$ implies $s\ge t$. 

\begin{figure}[t]
   \centering
	\vspace{5pt}
   \begin{tabular}{@{}c@{\hspace{.5cm}}c@{}}
       \includegraphics[width=.5\linewidth]{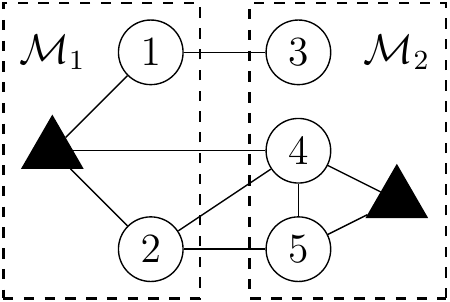}
   \end{tabular}
 \caption{An example of the interconnection of two microgrids $\mc M_1$, $\mc M_2$ satisfying Assumption \ref{as}. Source nodes are represented by $\blacktriangle$ and load nodes are represented by $\bigcirc$.}
 \label{example figure}
\end{figure}
\begin{example}
Figure \ref{example figure} represents an interconnection of two microgrids for which Assumption \ref{as} holds: Nodes 1, 2, 4 and 5 satisfy \ref{as}.i, while nodes 3, 4 and 5 satisfy \ref{as}.ii. Note that node 3 is not path-connected with respect to $\Gamma_{\mc M_2}$ to a source in $\mc M_2$. This implies that Assumption \ref{as} would not hold for node 3 if the numbering of $\mc M_1$ and $\mc M_2$ would be interchanged. Put differently, $\mc M_2$ cannot operate in island mode. 
\end{example}

Assumption 5 implies the following Lemma.
\begin{lemma}\label{example for first assumption}
The matrix $Y_{LL}$ is an invertible M-matrix.
\end{lemma}
\begin{proof}
It follows from Assumption \ref{as} that every load node is path-connected to a source node. This implies that $Y_{LL}$ is weakly chained diagonally dominant, and hence an invertible M-matrix by Corollary 4 of \cite{paper4}.
\end{proof}


%
%
%
%

\subsection{The block Cholesky decomposition}\label{bcd section}
This section reviews the block Cholesky decomposition (BCD) of a symmetric positive definite matrix. More specifically, we consider the BCD of M-matrices.

Consider a symmetric positive definite matrix $A\in\RR^{n\times n}$, along with a partition of the rows and columns of $A$ into $k$ nonempty subsets. We obtain $A = \begin{pmatrix}
a_{11} & \cdots & a_{1k}\\
\vdots & & \vdots\\
a_{k1} & \cdots & a_{kk}
\end{pmatrix}$ and $a_{ij} \in \RR^{n_i\times n_j}$ such that $\sum_{i=1}^k\limits n_i=n$.
Note that $a_{ij} = a_{ji}\T$.
We define the matrices
\begin{align*}
A_i &:= \begin{pmatrix}
a_{11} & \cdots & a_{1i}\\
\vdots & & \vdots\\
a_{i1} & \cdots & a_{ii}
\end{pmatrix} \text{ for }i=1,\sdots,k;\\
\alpha_i &:= \begin{pmatrix}
a_{i1} & \cdots & a_{i,i-1}
\end{pmatrix} \text{ for }i=2,\sdots,k.
\end{align*}
\vspace{-5pt}\\
We have $A= A_k$ and $A_i = \begin{pmatrix}
A_{i-1} & \alpha_i\T \\ \alpha_i & a_{ii}
\end{pmatrix}$ for $i=2,\sdots,k$.

Let the matrices $C_i, D_i$ for $i=1,\sdots,k$ be iteratively defined by
\begin{align*}
&C_1 := I_{n_1}, &C_i &:= \begin{pmatrix}
C_{i-1} & 0\\
\alpha_i A_{i-1}\inv C_{i-1} & I_{n_i}
\end{pmatrix}, \\
&D_1:= A_{11}, &D_i &:= \begin{pmatrix}
D_{i-1} & 0 \\ 0 & a_{ii} - \alpha_i A_{i-1}\inv \alpha_i\T
\end{pmatrix}.
\end{align*}
It follows that $A_i = C_i D_i C_i\T$ for $i=1,\sdots,k$.\\
Let $C:= C_k$, $D:= D_k$, then the \emph{block Cholesky decomposition} (BCD) of $A$ is given by $A = C D C\T$.

Note that
\begin{align*}
&C_i = \begin{pmatrix}
C_{i-1} & 0\\
\alpha_i A_{i-1}\inv C_{i-1} & I_{n_i}
\end{pmatrix}=\begin{pmatrix}
C_{i-1} & 0\\
\alpha_i C_{i-1}\invT D_{i-1}\inv & I_{n_i}
\end{pmatrix};\\
&C_i\inv = \begin{pmatrix}
C_{i-1}\inv & 0\\
-\alpha_i A_{i-1}\inv & I_{n_i}
\end{pmatrix}=\begin{pmatrix}
C_{i-1}\inv & 0\\
-\alpha_i C_{i-1}\invT D_{i-1}\inv C_{i-1}\inv & I_{n_i}
\end{pmatrix}
\end{align*}
from which it can be observed that only the diagonal blocks of $D$ have to be inverted to obtain the BCD of $A$.

By taking $k=n$ we obtain the Cholesky decomposition of $A$, and the BCD therefore generalizes the Cholesky decomposition.

Suppose that $A$ is an invertible M-matrix, then $\alpha_i\le 0$ and $A_i\inv \ge 0$. Since $a_{ii}-\alpha_iA_{i-1}\inv\alpha_i\T$ is a Schur complement of the invertible M-matrix $A_i$, it is again an invertible M-matrix, and therefore $(a_{ii}-\alpha_iA_{i-1}\inv\alpha_i\T)\inv \ge 0$. Using this fact, it can be shown by induction on $i$ that $D_i\inv \ge 0$ and $C_i\inv \ge 0$. Since both are invertible, it follows that $D_i\inv$ and $C_i\inv$ are order-preserving for all $i$.

\subsection{Iterative feasibility condition using the BCD}\label{iterative section}
Continuing Section \ref{system description}, we define $A:=Y_{LL}$. We let $C,D$ such that $A=CDC\T$ is the BCD of $A$ where each block corresponds to a microgrid in the power grid. We will use the notation for submatrices of $A$ as in Section \ref{bcd section}.

The intuition behind the BCD of $Y_{LL}$ is not straight-forward, but can be seen as follows.
Each diagonal block of $D$ describes the conductances between load nodes in a microgrid after all load nodes which are lower in the hierarchy are eliminated by Kron reduction.
Put differently, for each block it is assumed that there is no current flow at the load nodes which are lower in hierarchy, which means that these nodes do not consume power.

The matrix $C$ describes all weighted paths between load nodes in distinct microgrids which ascend the hierarchy, and is therefore (block) lower triangular. The off-diagonal elements of $C$ are nonnegative, while $C$ has ones on its diagonal.
The matrix $CD$ roughly represents the conductances between load nodes in a microgrid, under the assumption that load nodes which are lower in hierarchy draw no current, compensated by the conductances between load nodes in microgrids which are lower in hierarchy and have the same assumption. This compensation is based on paths which descend the hierarchy, giving $CD$ also a block lower triangular structure.

%
%
%


\begin{lemma}\label{path inverse}
The load nodes $k$ and $l$ in $\mc M_i$ are path-connected with respect to the graph induced by the load nodes in $\mc M_1,\sdots,\mc M_i$ if and only if $(a_{ii}-\alpha_iA_{i-1}\inv \alpha_i\T)\inv_{kl} > 0$.
\end{lemma}
\begin{proof}
The connected components of the graph induced by the load nodes in $\mc M_1,\sdots,\mc M_i$ correspond to the irreducible components of $A_i$. We can permute the rows and columns of $A_i$ such that $A_i$ is block-diagonal. By the Perron-Frobenius theorem for irreducible matrices, the inverse of each block is positive. It follows that the $(k,l)$-entry of $A_i\inv$ is positive if and only if $k$ and $l$ are path-connected with respect to the graph. It can be shown that $(a_{ii}-\alpha_iA_{i-1}\inv \alpha_i\T)\inv$ is the principal submatrix of $A_i\inv$ corresponding to load nodes of $\mc M_i$.
\end{proof}

The next lemma is the cornerstone of the main theorem in this section, and the reason why we need Assumption \ref{as}.
\begin{lemma}\label{positive vector}
Suppose Assumption \ref{as} holds.
The vector $C\T V_L^*$ is positive.
\end{lemma}
\begin{proof}
Recall that 
\begin{align*}
V_L^* = - A\inv Y_{LS} V_S = - C\invT D\inv C\inv Y_{LS} V_S,
\end{align*}
and therefore $C\T V_L^* = - D\inv C\inv Y_{LS} V_S$. Note that $V_S>\mb 0$ and $-Y_{LS}\ge 0$, and therefore $-Y_{LS}V_S\ge \mb 0$.\footnote{All load nodes share an edge with a generator node if and only if $-Y_{LS}V_S>\mb 0$.} Since $A$ is an invertible M-matrix, we have seen that $C\inv \ge 0$ and $D\inv \ge 0$. This implies that $- D\inv C\inv Y_{LS} V_S\ge 0$.

We consider the rows of $C\T V_L^*$ corresponding to the load nodes in $\mc M_i$, which are given by \begin{align}\label{positive argument}
-(a_{ii} - \alpha_i A_{i-1}\inv \alpha_i\T)\inv \begin{pmatrix}
-\alpha_i A_{i-1}\inv & I_{n_i} & 0
\end{pmatrix} Y_{LS} V_S.
\end{align}
Consider a load node $k$ in $\mc M_i$.

If Assumption \ref{as}.i holds for $k$, then $k$ is path-connect to a node $l$ in $\mc M_i$ with respect to the graph $\Gamma_{\mc M_i}$ (or coincides with $l$) such that $l$ shares an edge with a source node in $\mc M_i$. This implies that $(a_{ii} - \alpha_i A_{i-1}\inv \alpha_i\T)\inv_{kl} (-Y_{LS} V_S)_l > 0$ by Lemma \ref{path inverse} and since $ (-Y_{LS} V_S)_l > 0$ if $l$ shares an edge with a source node.

If Assumption \ref{as}.ii holds for $k$, then $k$ is path-connected to a node $l$ in $\mc M_j$ with respect to $\Gamma$ such that $j<i$, and $l$ is the only node in the path not in $\mc M_i$. 

If $l$ is a source node, then again $(a_{ii} - \alpha_i A_{i-1}\inv \alpha_i\T)\inv_{kl^\prime} (-Y_{LS} V_S)_{l^\prime} > 0$ where $l^\prime$ is the node in the path which shares an edge with $l$.

If $l$ is a load node, then, by applying induction on $i$ in the above, Assumption \ref{as} implies that $l$ is path-connected to a load node $m$ in $\mc M_{j^\prime}$ such that ${j^\prime}\le j$, which shares an edge with a source node. Hence 
\begin{align*}
(-(a_{ii} - \alpha_i A_{i-1}\inv \alpha_i\T)\inv\alpha_i)_{kl}(A_{i-1}\inv)_{lm}(-Y_{LS} V_S)_m > 0
\end{align*}
by Lemma \ref{path inverse}. This implies that each row of (\ref{positive argument}) is positive.
\end{proof}

The vector $C\T V_L^*$ is a lower bound for the open circuit voltages $V_L^*$ and roughly represents the effect of the potentials at the sources when there is no current flow at loads and where currents only flow over paths which descend the hierarchy.

The next theorem gives a sufficient condition for feasibility which is more conservative then Theorem \ref{SP result}, but incorporates the topological structure of the microgrids.

\begin{theorem}\label{chol theorem}
Let $Y_{LL}=CDC\T$, the block Cholesky decomposition such that the blocks (\ie microgrids) satisfy Assumption \ref{as}.
If 
\begin{align}\label{chol theorem condition}
D\inv C\inv [C\T V_L^*]\inv P_L < \frac 1 4 C\T V_L^*,
\end{align}
then $Y_{LL}\inv [V_L^*]\inv P_L < \frac 1 4 V_L^*$ and $(Y_{LL},-Y_{LS}V_S,P_L)$ is feasible.
\end{theorem}

\begin{proof}
From Lemma \ref{positive vector} it follows that $[C\T V_L^*]\inv $ is well-defined and nonnegative. The matrix $C\invT$ has ones on its diagonal, and since $C\invT$ is nonnegative, it follows that $I \le C\invT$. It follows that $C\T V_L^* \le C\invT C\T V_L^* = V_L^*$ and so $[C\T V_L^*]\inv P_L \ge [V_L^*]\inv P_L$. The matrix $D\inv C\inv$ is order-preserving and so, 
\begin{align*}
D\inv C\inv [V_L^*]\inv P_L \le D\inv C\inv [C\T V_L^*]\inv P_L < \frac 1 4 C\T V_L^*.
\end{align*}
Multiplication by the order-preserving matrix $C\invT$ yields
\begin{align*}
C\invT D\inv C\inv [V_L^*]\inv P_L = Y_{LL}\inv [V_L^*]\inv P_L < \frac 1 4 V_L^*.
\end{align*}
Theorem \ref{SP result} implies that $(Y_{LL},-Y_{LS}V_S,P_L)$ is feasible.
\end{proof}
\begin{figure}[b]
   \centering
   \begin{tabular}{@{}c@{\hspace{.5cm}}c@{}}
       \includegraphics[width=0.95\linewidth]{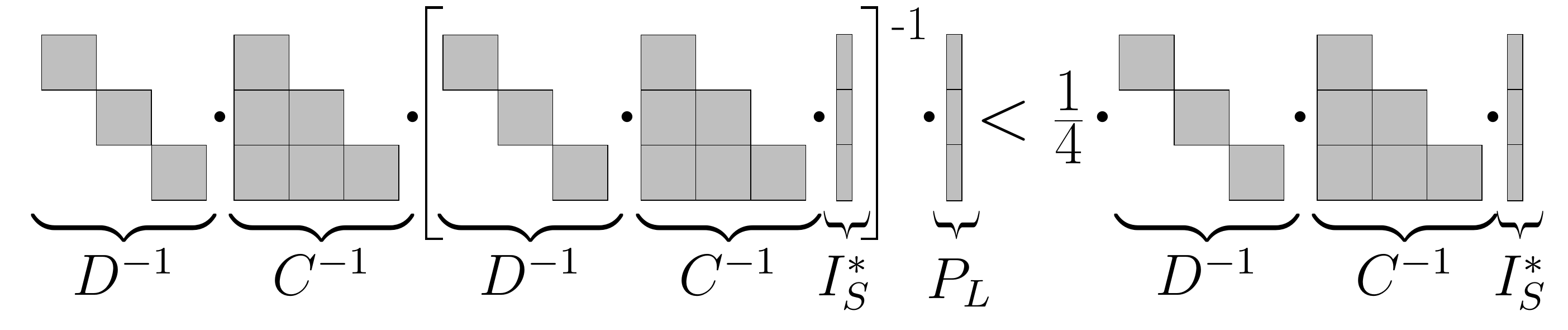}
   \end{tabular}
\caption{A schematic representation of inequality \eqref{chol theorem condition} and \eqref{chol theorem condition rewritten} for three microgrids, where $I_S^* := -Y_{LS}V_S$. Due to the block lower triangular structure of $C\inv$ and $D\inv$, 
the inequality concerning the first two block rows is independent of the third block row of matrices $D\inv$, $C\inv$ and vectors $I_S^*$, $P_L$. This suggests that the introduction of a fourth block row to $D\inv$ and $C\inv$ will therefore not compromise the inequality of the first three block rows.}
\label{schematic condition figure}
\end{figure}
Theorem \ref{chol theorem} implies that if \eqref{chol theorem condition} holds, both $(C D,-Y_{LS}V_S,P_L)$ and $(Y_{LL},-Y_{LS}V_S,P_L)$ are feasible. More generally, it can be shown that $(C D,-Y_{LS}V_S,P_L)$ is feasible only if $(Y_{LL},-Y_{LS}V_S,P_L)$ is feasible.


The condition \eqref{chol theorem condition} in Theorem \ref{chol theorem} is equivalent to
\begin{align}\label{chol theorem condition rewritten}
D\inv C\inv [D\inv C\inv I_S^*]\inv P_L < \tfrac 1 4 D\inv C\inv I_S^*, 
\end{align}
where $I_S^* := -Y_{LS}V_S$, the current from sources flowing to neighboring loads. 
The block structure of \eqref{chol theorem condition rewritten} is schematically represented by Figure \ref{schematic condition figure}.

The introduction of a new microgrid gives rise to a new block row of $C\inv$ and $D\inv$.
Figure \ref{schematic condition figure} illustrates a plug-and-play condition such that, if a power grid satisfies \eqref{chol theorem condition}, and a microgrid satisfies a condition similar to the last block row in Figure \ref{schematic condition figure}, then their interconnection satisfies \eqref{chol theorem condition} as well.
However, there are some technicalities which prevent us from directly applying Theorem \ref{chol theorem} to formulate a plug-and-play condition for feasibility. This is investigated further in Section \ref{dynamic section}.


\section{Microgrids with a dynamic\break interconnection topology}\label{dynamic section}
In this section we relate the feasibility of a power grid to an associated ``virtual power grid'', such that Theorem \ref{chol theorem} can be applied. From this, a plug-and-play condition for feasibility is obtained, which is the main result of this paper.

\begin{example}\label{diagonal example}
Consider the power grid in Figure \ref{example figure} where we assume unit line conductances. 
We consider the microgrid $\mc M_1$ in island mode and let $Y$ represent its Laplacian matrix. We have $Y_{LL} = \begin{pmatrix}
1 & 0 \\ 0 & 1
\end{pmatrix}$.
We proceed by connecting $\mc M_2$ to $\mc M_1$ in the same fashion as in Figure \ref{example figure}, and let $\hat Y$ represent the Laplacian matrix of the interconnected microgrids. It follows that 
\begin{align*}
\hat Y_{LL} = \l(\begin{array}{cc|ccc}
2 & 0 & -1 & 0 & 0 \\ 
0 & 3 & 0 & -1 & -1 \\
\hline
-1 & 0 & 1 & 0 & 0 \\
0 & -1 & 0 & 4 & -1 \\
0 & -1 & 0 & -1 & 3
\end{array}\r),
\end{align*}
where the partitioning of the matrix correspond to the two microgrids. 
Note that the diagonal elements corresponding to node 1 and node 2 have increased.
\end{example}

%
%
%
%
%

To formulate a plug-and-play condition as suggested in Section \ref{iterative section}, it is
essential for Theorem \ref{chol theorem} that the blocks of the BCD of $Y_{LL}$ remain unaltered when new microgrids are attached to the power grid.
However, from Example \ref{diagonal example} we see that, when new lines are introduced, the diagonal elements of the matrix $\hat Y_{LL}$ change with respect to $Y_{LL}$.

\subsection{Virtual shunts and virtual power grids}\label{virtual section}
To address the issue above, we increase the diagonal elements of the Laplacian by temporarily introducing shunts\footnote{A shunt is a component at a node which extracts current proportional to the voltage potential at the node.} at load nodes. These shunts are later decreased (and potentially removed) as more lines are interconnected to a load node.
We refer to these shunts as \emph{virtual shunts}. 
The result is that the diagonal elements of the Laplacian matrix remain constant when new microgrids are attached.

Virtual shunts are not physically present in the power grid and should be seen as placeholders for prospective lines. 
We refer to a power grid with virtual shunts as a \emph{virtual power grid}. 
We refer to other power grids as physical power grids, to prevent ambiguity.
It should be understood that a virtual power grid does not fully capture the behavior of the associated physical power grid. However, the feasibility of the two is related by the following lemma.

\begin{lemma}\label{virtual lemma}
Let $A$ be an invertible M-matrix, $E\ge 0$ diagonal, $b\ge \mb 0$ and $c>\mb 0$ such that $(A+E)\inv b > \mb 0$.
Suppose $(A+E)\inv [(A+E)\inv b]\inv c < \frac 1 4 (A+E)\inv b$, then $A\inv [A\inv b]\inv c < \frac 1 4 A\inv b$.
\end{lemma}
\begin{proof}
Note that $A\inv(A+E) = I + A\inv E \ge I$ is order-preserving, which implies that $A\inv b \ge (A+E)\inv b$, and so $[A\inv b]\inv \le [(A+E)\inv b]\inv$. By multiplying our condition with $A\inv(A+E)$, we obtain
{\small$A\inv [(A+E)\inv b]\inv c < \tfrac 1 4 A\inv b$}. The inequality $[A\inv b]\inv \le [(A+E)\inv b]\inv$ implies that 
$A\inv [A\inv b]\inv c < \frac 1 4 A\inv b$, as $A\inv$ is order-preserving.
\end{proof}
In the context of Theorem \ref{SP result}, Lemma \ref{virtual lemma} implies that if $(Y_{LL}+E,Y_{LS}V_S,P_L)$ with  diagonal $E\ge0$ satisfies the condition of Theorem \ref{SP result}, then so does $(Y_{LL},Y_{LS}V_S,P_L)$, implying that both are feasible. 
The power grid corresponding to $(Y_{LL} + E,Y_{LS}V_S,P_L)$ is a virtual counterpart to $(Y_{LL},Y_{LS}V_S,P_L)$, with $E$ representing the virtual shunts.
More generally, it can be shown that $(Y_{LL}+E,Y_{LS}V_S,P_L)$ is feasible only if $(Y_{LL},Y_{LS}V_S,P_L)$ is feasible.

Note that virtual shunts are not needed for source nodes. This follows directly from the observation that \eqref{main problem} does not appear in $Y_{SS}$. Hence, lines between sources can be altered freely, without compromising feasibility.

It is essential that the virtual shunts are chosen properly.
Choosing $E$ large leads to more conservative conditions, whereas choosing $E$ small restricts the possible interconnections with potential microgrids.

However, from a practical point of view, it is reasonable to assume that the conductances of prospective lines are known \textit{a priori}. For example, some lines may already exist but are not in use. Also, there might be limitations on the number of lines a node can connect with, which would give an upper bound for the virtual shunt. In particular, some nodes might never connect to nodes outside their microgrid and do not require a virtual shunt.

\subsection{Main Theorem: Plug-and-play solvability}\label{main theorem section}
In this section we present the main result of the paper. The main line of this section is as follows.

We consider a power grid and define a virtual counterpart in the sense of Section \ref{virtual section}.
We define a BCD of the lines of the loads in the virtual power grid.
We assume that the BCD satisfies the conditions for Theorem \ref{chol theorem}.
This implies that both the power grid and its virtual counterpart are feasible, by Lemma \ref{virtual lemma}.

We proceed by introducing a microgrid and again define a virtual counterpart.
We interconnect both virtual grids under the assumption that the physical interconnection satisfies Assumption \ref{as} and does not exceed the virtual shunts.
We extend the BCD of the virtual power grid to this virtual interconnection and use this to state the main theorem.

The main theorem gives a sufficient condition such that this interconnected virtual power grid is feasible. Moreover, the theorem states that the physical interconnection of the power grid and the microgrid is therefore also feasible, by Lemma \ref{virtual lemma}.\\

Consider a DC power grid described by the Laplacian matrix
$\begin{pmatrix}
Y_{L_1L_1} & Y_{L_1S_1}\\
Y_{S_1L_1} & Y_{S_1S_1}
\end{pmatrix}$,
together with the virtual shunts $E_{L_1}$ and load voltages $V_{L_1}$.
Let $P_{L_1}$ be the power demand at the load nodes and $V_{L_1}^\prime := -(Y_{L_1L_1} + E_{L_1})\inv Y_{L_1S_1}V_{S_1}$ the virtual open-circuit voltages, where $V_{S_1}$ are the source voltages. The vector $V_{L_1}^\prime$ is a lower bound for the open-circuit voltages of the physical power grid.

Let $C,D$ such that $CDC\T = Y_{L_1L_1} + E_{L_1}$ is a BCD, where the blocks are such that they satisfy Assumption \ref{as}.
\begin{assumption}\label{assumption to main theorem}
The following inequality holds:
\begin{align*}
D\inv C\inv[C\T V_{L_1}^\prime]\inv P_{L_1} < \tfrac 1 4 C\T V_{L_1}^\prime.
\end{align*}
\end{assumption}
Assumption \ref{assumption to main theorem} implies that $(Y_{L_1L_1}+E_{L_1}, Y_{L_1S_1}V_{S_1}, P_{L_1})$ is feasible by Theorem \ref{chol theorem}. It follows by Lemma \ref{virtual lemma} that
$(Y_{L_1L_1},-Y_{L_1S_1}V_{S_1},P_{L_1})$, the physical counterpart, is feasible as well.
\\

In addition, consider a microgrid described by the Laplacian matrix
$\begin{pmatrix}
Y_{L_2L_2} & Y_{L_2S_2}\\
Y_{S_2L_2} & Y_{S_2S_2}
\end{pmatrix}$, together with the virtual shunts $E_{L_2}$ and load voltages $V_{L_2}$.
Let $P_{L_2}$ be the power demand at the load nodes.

Let the physical interconnection of the power grid and the microgrid be given by the Laplacian matrix 
\begin{align*}
\left(\begin{smallmatrix}
Y_{L_1L_1} + \hat E_{L_1} & Y_{L_1L_2} & Y_{L_1S_1} & Y_{L_1S_2}\\
Y_{L_2L_1} & Y_{L_2L_2} + \hat E_{L_2} & Y_{L_2S_1} & Y_{L_2S_2}\\
Y_{S_1L_1} & Y_{S_1L_2} & Y_{S_1S_1} + \hat E_{S_1} & Y_{S_1S_2}\\
Y_{S_2L_1} & Y_{S_2L_2} & Y_{S_2S_1} & Y_{S_2S_2} + \hat E_{S_2}
\end{smallmatrix}\right),
\end{align*}
where $\hat E_{L_1}$, $\hat E_{L_2}$, $\hat E_{S_1}$, $\hat E_{S_2}$ are defined as follows:
\begin{align*}
\hat E_{L_1} :=-[Y_{L_1L_2}\mb 1 + Y_{L_1S_2}\mb 1]; \hat E_{S_1} :=-[Y_{S_1L_2}\mb 1 + Y_{S_1S_2}\mb 1];\\
\hat E_{L_2} := -[Y_{L_2L_1}\mb 1 + Y_{L_2S_1}\mb 1]; \hat E_{S_2} :=-[Y_{S_2L_1}\mb 1 + Y_{S_2S_1}\mb 1].
\end{align*}

We make the following assumption on the interconnection.
\begin{assumption}\label{interconnection assumption}
The interconnection of the power grid and the microgrid is such that $\hat E_{L_1} \le E_{L_1}$, $\hat E_{L_2} \le E_{L_2}$ and Assumption \ref{as} holds.
\end{assumption}

The bounds on $\hat E_{L_1}$ and $\hat E_{L_2}$ imply that the virtual interconnection of both virtual grids does not exceed the virtual shunts. This will allow us to use Lemma \ref{virtual lemma} in Theorem \ref{main theorem}.

Let $\tilde C,\tilde D$ such that 
\begin{align*}
\tilde C\tilde D\tilde C\T = \begin{pmatrix}
Y_{L_1L_1} + E_{L_1} & Y_{L_1L_2} \\
Y_{L_2L_1} & Y_{L_2L_2} + E_{L_2} 
\end{pmatrix}
\end{align*} is a BCD. From Section \ref{bcd section} it follows that 
\begin{align*}
\tilde C\inv &= \begin{pmatrix}
C\inv & 0 \\
-Y_{L_2L_1} (CDC\T)\inv & I_n 
\end{pmatrix};\\
\tilde D\inv &= \begin{pmatrix}
D\inv & 0 \\ 0 & R\inv
\end{pmatrix}, 
\end{align*}
where we define 
\begin{align*}
R:= Y_{L_2L_2} + E_{L_2} - Y_{L_2L_1} (CDC\T)\inv Y_{L_1L_2}.
\end{align*}

Define $V_L^\prime := $
$-(\tilde C\tilde D\tilde C\T )\inv \l(\begin{smallmatrix}
Y_{L_1S_1} & Y_{L_1S_2}\\
Y_{L_2S_1} & Y_{L_2S_2}
\end{smallmatrix}\r)\l(\begin{smallmatrix}
V_{S_1}\\
V_{S_2}
\end{smallmatrix}\r)$, which is
a lower bound for the open-circuit voltages of the interconnected power grid.

Finally, let $(\tilde C\T V_L^\prime)_{L_1}$ denote the entries of $\tilde C\T V_L^\prime$ corresponding to the load nodes of the original power grid, and similar for $(\tilde C\T V_L^\prime)_{L_2}$ and the microgrid. 

\begin{theorem}[Plug-and-Play Solvability]\label{main theorem}
Suppose that Assumptions \ref{assumption to main theorem} and \ref{interconnection assumption} are satisfied.
If the condition
\begin{multline}\label{main theorem condition}
-R\inv Y_{L_2L_1} (CDC\T)\inv [(\tilde C\T V_L^\prime)_{L_1}]\inv P_{L_1}\\ + R\inv [(\tilde C\T V_L^\prime)_{L_2}]\inv P_{L_2} < \tfrac 1 4 (\tilde C\T V_L^\prime)_{L_2}
\end{multline}
holds, then
\begin{align}\label{main theorem final condition}
\tilde D\inv \tilde C\inv[\tilde C\T V_L^\prime]\inv \l(\begin{smallmatrix} P_{L_1} \\ P_{L_2} \end{smallmatrix}\r) < \tfrac 1 4 \tilde C\T V_L^\prime
\end{align}
and the interconnection of the power grid and the microgrid is feasible.

\end{theorem}
\begin{proof}
Due to the block triangular structure of $\tilde C$, the rows of (\ref{main theorem final condition}) corresponding to the virtual power grid are given by 
\begin{align}\label{main theorem first block}
D\inv C\inv[(\tilde C\T V_L^\prime)_{L_1}]\inv P_{L_1} < \tfrac 1 4 (\tilde C\T V_L^\prime)_{L_1}.
\end{align} 
For the same reason we have 
\begin{align*}
(\tilde C\T V_L^\prime)_{L_1} &= - D\inv C\inv (Y_{L_1S_1}V_{S_1} + Y_{L_1S_2}V_{S_2})\\
& \ge  -D\inv C\inv Y_{L_1S_1}V_{S_1} = C\T V_{L_1}^\prime,
\end{align*}
which implies $[(\tilde C\T V_L^\prime)_{L_1}]\inv \le [C\T V_{L_1}^\prime]\inv$.
Assumption \ref{assumption to main theorem} implies that 
\begin{align*}
D\inv C\inv[(\tilde C\T V_L^\prime)_{L_1}]\inv P_{L_1} &\le D\inv C\inv[C\T V_{L_1}^\prime]\inv P_{L_1} \\
< \tfrac 1 4 C\T V_{L_1}^\prime &\le \tfrac 1 4 (\tilde C\T V_L^\prime)_{L_1}.
\end{align*}
Therefore (\ref{main theorem first block}) is satisfied.

Note that (\ref{main theorem condition}) is equivalent to the rows of (\ref{main theorem final condition}) corresponding to the virtual microgrid. Hence if (\ref{main theorem condition}) holds, then so does (\ref{main theorem final condition}).

If (\ref{main theorem final condition}) is satisfied, then by Theorem \ref{chol theorem} the virtual interconnection \begin{align*}
(\l(\begin{smallmatrix}
Y_{L_1L_1} + E_{L_1} & Y_{L_1L_2} \\
Y_{L_2L_1} & Y_{L_2L_2} + E_{L_2} 
\end{smallmatrix}\r),\l(\begin{smallmatrix}
Y_{L_1S_1} & Y_{L_1S_2}\\
Y_{L_2S_1} & Y_{L_2S_2}
\end{smallmatrix}\r)\l(\begin{smallmatrix}
V_{S_1}\\
V_{S_2}
\end{smallmatrix}\r),\l(\begin{smallmatrix} P_{L_1} \\ P_{L_2} \end{smallmatrix}\r))
\end{align*}
is feasible. Since $\hat E_{L_1} \le E_{L_1}$ and $\hat E_{L_2} \le E_{L_2}$, we may use Lemma \ref{virtual lemma} to conclude that the physical interconnection
\begin{align*}
(\l(\begin{smallmatrix}
Y_{L_1L_1} + \hat E_{L_1} & Y_{L_1L_2} \\
Y_{L_2L_1} & Y_{L_2L_2} + \hat E_{L_2} 
\end{smallmatrix}\r),\l(\begin{smallmatrix}
Y_{L_1S_1} & Y_{L_1S_2}\\
Y_{L_2S_1} & Y_{L_2S_2}
\end{smallmatrix}\r)\l(\begin{smallmatrix}
V_{S_1}\\
V_{S_2}
\end{smallmatrix}\r),\l(\begin{smallmatrix} P_{L_1} \\ P_{L_2} \end{smallmatrix}\r))
\end{align*}
is feasible as well.
\end{proof}

Note that Assumption \ref{assumption to main theorem} and \eqref{main theorem final condition} are of the same form. Hence, microgrids can be attached to a power grid in an iterative fashion by using Theorem \ref{main theorem} to guarantee feasibility. This provides a method to determine if a power grid remains feasible after connecting a specific microgrid

\subsection{Example of plug-and-play solvability}
The following example illustrates the application of Theorem \ref{main theorem} to determine if a power grid remains feasible after interconnecting a specific microgrid.

\begin{example}
Consider again the power grid in Figure \ref{example figure} with unit line conductance, where we connect $\mc M_2$ to the power grid $\mc M_1$. 
We continue in the notation of Section \ref{main theorem section}.

Let the voltage potentials at each source be 1 volt, and consider constant-power loads such that $P_{L_1}=\frac 1 {25} \begin{pmatrix}
2& 2
\end{pmatrix}\T$ and $P_{L_2}=\frac 1 {25} \begin{pmatrix}
1& 9& 7
\end{pmatrix}\T$.

For microgrids $\mc M_1$ and $\mc M_2$ we have \begin{align*}
Y_{L_1L_1} &= \begin{pmatrix}
1 & 0 \\ 0 & 1
\end{pmatrix};&&Y_{L_2L_2}= \begin{pmatrix}
0 & 0 & 0\\ 
0 & 2 & -1\\
0 & -1 & 2
\end{pmatrix}.
\end{align*}
We consider the virtual shunts
\begin{align*}
E_{L_1}=\begin{pmatrix}
1 & 0 \\ 0 & 2
\end{pmatrix};&&E_{L_2}=\begin{pmatrix}
1 & 0 & 0\\ 
0 & 2 & 0\\
0 & 0 & 1
\end{pmatrix}.
\end{align*}

We have $C=\begin{pmatrix}
1 & 0 \\ 0 & 1
\end{pmatrix}$ and $D=Y_{L_1L_1} + E_{L_1} = \begin{pmatrix}
2 & 0 \\ 0 & 3
\end{pmatrix}$.\linebreak
The virtual open-circuit voltages for $\mc M_1$ are $V_{L_1}^\prime = \begin{pmatrix}
\frac 1 2 &\frac 1 3
\end{pmatrix}\T$. It follows that $C\T V_{L_1}^\prime = \begin{pmatrix}
\frac 1 2 &\frac 1 3
\end{pmatrix}\T$ and 
\begin{align*}
&D\inv C\inv[C\T V_{L_1}^\prime]\inv P_L = \begin{pmatrix}
\frac 1 2 & 0 \\ 0 & \frac 1 3
\end{pmatrix} \begin{bmatrix}
\begin{pmatrix}
\frac 1 2 \\ \frac 1 3
\end{pmatrix}
\end{bmatrix}\inv \begin{pmatrix}
\frac 2 {25} \\ \frac 2 {25}
\end{pmatrix} \\
&\quad= \begin{pmatrix}
\frac 2 {25} \\ \frac 2 {25}
\end{pmatrix} < \begin{pmatrix}
\frac 1 8 \\ \frac 1 {12}
\end{pmatrix} = \tfrac 1 4 \begin{pmatrix}
\frac 1 2 \\ \frac 1 3
\end{pmatrix} = \tfrac 1 4 C\T V_{L_1}^\prime.
\end{align*}
Hence Assumption \ref{assumption to main theorem} is satisfied and $\mc M_1$ is feasible in island mode.

We consider the interconnection in Figure \ref{example figure}, which satisfies $\hat E_{L_2} = E_{L_2}$ and $\hat E_{L_2} = E_{L_2}$. Since Assumption \ref{as} holds for the interconnection, Assumption \ref{interconnection assumption} is satisfied.

We compute $V_L^\prime = \begin{pmatrix}
1 & 1 & 1 & 1 & 1
\end{pmatrix}\T$ and $\tilde C\T V_L^\prime = \begin{pmatrix}
\frac 1 2 & \frac 1 3 & 1 & 1 & 1
\end{pmatrix}\T$. 
Note that $Y_{L_2L_1}=\begin{pmatrix}
-1 & 0 \\
0 & -1 \\
0 & -1 
\end{pmatrix}$.

We compute $R\inv = \begin{pmatrix} 
2 & 0 & 0 \\
0 & \frac 1 3 & \frac 1 6\\
0 & \frac 1 6 & \frac {11} {24}
\end{pmatrix}$. Continuing from the left-hand side of \eqref{main theorem condition}, we have 
\begin{align*}
-\begin{pmatrix} 
2 & 0 & 0 \\
0 & \frac 1 3 & \frac 1 6\\
0 & \frac 1 6 & \frac {11} {24}
\end{pmatrix}\begin{pmatrix}
-1 & 0 \\
0 & -1 \\
0 & -1 
\end{pmatrix}\begin{pmatrix}
\frac 1 2 & 0 \\ 0 & \frac 1 3
\end{pmatrix}\begin{bmatrix}\begin{pmatrix}
\frac 1 2 \\ \frac 1 3
\end{pmatrix}\end{bmatrix}\inv \begin{pmatrix}
\frac 2 {25} \\ \frac 2 {25}
\end{pmatrix}\\
+\begin{pmatrix} 
2 & 0 & 0 \\
0 & \frac 1 3 & \frac 1 6\\
0 & \frac 1 6 & \frac {11} {24}
\end{pmatrix} \begin{bmatrix}\begin{pmatrix}
1 \\ 1 \\ 1
\end{pmatrix}\end{bmatrix}\inv \begin{pmatrix}
\frac 1 {25} \\ \frac 9 {25} \\ \frac 7 {25}
\end{pmatrix} \\
\approx \begin{pmatrix}
0.24\\0.21\\0.24
\end{pmatrix} < \tfrac 1 4\begin{pmatrix}
1\\1\\1
\end{pmatrix} = \tfrac 1 4 (C\T V_L^\prime)_{L_2}.
\end{align*}
Hence, the interconnection of the two microgrids is feasible, by Theorem \ref{main theorem}.
\end{example}

\section{Conclusion and future research}\label{conclusion section}
In this paper we studied the interconnection of purely resistive DC microgrids with constant-power loads. We have presented a sufficient condition for the feasibility of a DC power grid. In addition, we have presented a sufficient condition for the interconnection of a microgrid such that the former sufficient condition also holds for their interconnection. This establishes a method to determine if a power grid remains feasible after connecting a specific microgrid. The method was illustrated by an example.

The main result of the paper is directly derived from paper \cite{paper1} and is more conservative than the condition in \cite{paper1}. However, the novelty of the presented result is its plug-and-play property. It can be shown that the implications for feasibility presented in this paper are not specific to the condition in \cite{paper1} and hold in the general setting. Hence, alternative or improved sufficient conditions for feasibility may lead to other conditions with the plug-and-play property.

Other topics for future research may include more accurate models for DC power flow, improvement on voltage controller design for guaranteeing feasibility and a detailed analysis of the conservative nature of the presented result.

\end{document}